\documentclass[hidelinks,onefignum,onetabnum]{siamart220329}

\usepackage{lipsum}
\usepackage{amsfonts}
\usepackage{graphicx}
\usepackage{epstopdf}
\usepackage{algorithmic}
\usepackage{longtable}
\ifpdf
\DeclareGraphicsExtensions{.eps,.pdf,.png,.jpg}
\else
\DeclareGraphicsExtensions{.eps}
\fi

\newsiamremark{remark}{Remark}
\newsiamremark{hypothesis}{Hypothesis}
\newsiamremark{assumption}{Assumption}
\crefname{hypothesis}{Hypothesis}{Hypotheses}
\newsiamthm{claim}{Claim}

\def\<{\left\langle}
\def\>{\right\rangle}
\def\E{\mathcal{E}}
\def\M{\mathcal{M}}
\def\T{{\rm T}}
\def\D{{\rm D}}

\def\B{\mathcal{B}}

\def\I{\mathcal{I}}

\def\R{\mathbb{R}}

\def\Re{{\rm Retr}}

\def\g{{\rm grad}}
\def\G{\mathcal{G}}
\def\V{\mathcal{V}}
\def\N{\mathcal{N}}
\def\h{{\rm Hess}}
\def\S{\mathbb{S}}

\def\rr{{\rm rank}}

\def\dd{{\rm diag}}
\def\P{{\rm Proj}}

\def\O{\mathcal{O}}

\def\({\left(}
\def\){\right)}

\headers{A Riemannian DRSOM with application in SNL}{T. Tang, K.-C. Toh, N. Xiao and Y. Ye}

\title{A Riemannian dimension-reduced second order method with application in sensor network localization\thanks{Submitted to the editors DATE.
		\funding{The research of Kim-Chuan Toh and Nachuan Xiao is supported by  Academic Research Fund Tier 3 grant call (MOE-2019-T3-1-010). The research of Yinyu Ye is partially supported by National University of Singapore when he was visiting there.}}}

\author{Tianyun Tang\thanks{Department of Mathematics, National
		University of Singapore,  
		(\email{ttang@u.nus.edu}).}
	\and Kim-Chuan Toh\thanks{Department of Mathematics, and Institute of 
		Operations Research and Analytics, National
		University of Singapore,
		(\email{mattohkc@nus.edu.sg}).}
	\and Nachuan Xiao\thanks{The Institute of Operations Research and Analytics, National University of Singapore, (\email{xnc@lsec.cc.ac.cn}).}
	\and Yinyu Ye\thanks{Department of Management Science and Engineering, Stanford University, (\email{yinyu-ye@stanford.edu})}
}

\ifpdf
\hypersetup{
  pdftitle={An Example Article},
  pdfauthor={}
}
\fi
   
\usepackage{amsmath}

\def\O{\mathcal{O}}
\begin{document}

\maketitle

\begin{abstract}
In this paper, we propose a cubic-regularized Riemannian optimization method (RDRSOM), which partially exploits the second order information and achieves the iteration complexity of $\O(1/\epsilon^{3/2}).$ In order to reduce the per-iteration computational cost, we further propose a practical version of (RDRSOM), which is an extension of the well known Barzilai-Borwein method and achieves the iteration complexity of $\O(1/\epsilon^{2})$. We apply our method to solve a nonlinear formulation of the wireless sensor network localization problem whose feasible set is a Riemannian manifold that has not been considered in the literature before. Numerical experiments are conducted to verify the high efficiency of our algorithm compared to state-of-the-art Riemannian optimization methods and other nonlinear solvers. 
\end{abstract}



\section{Introduction}
\subsection{Riemannian optimization}\label{subsec-rieopt}
In this paper, we consider the following problem:
\begin{equation}\label{Prob_Rie}
\min\left\{ f(x):\ x\in \M\right\}.
\end{equation}
Here $f:\M\rightarrow \R$ is a sufficiently smooth function defined on the Riemannian manifold $\M$ that is embedded in a finite-dimensional {Euclidean} space $\E$. Problem (\ref{Prob_Rie}) has many applications including machine learning  \cite{arjovsky2016unitary,bansal2018can,huang2018orthogonal,lezcano2019trivializations}, scientific computing \cite{cambier2016robust,danaila2017computation,kressner2016preconditioned,steinlechner2016riemannian} and semidefinite programming \cite{boumal2016non,journee2010low,tang2023feasible,tang2023solving}. When the manifold constraint $x \in \M$ is dropped (i.e., $\M$ is chosen as the {Euclidean} space $\E$), the optimization problem \eqref{Prob_Rie} has been extensively studied with a great number of efficient algorithms proposed such as conjugate gradient method, Newton method, quasi-Newton method, and trust-region method. Among these approaches, the momentum-accelerated method is a family of simple and efficient algorithm, which is proven to achieve better convergence rate for convex problems \cite{beck2009fast, nesterov1983method}, and  gains popularity in accelerating the stochastic optimization algorithms,  such as SGD \cite{sutskever2013importance} and Adam \cite{kingma2014adam}. In particular Castera et al. \cite{castera2021inertial} propose a class of momentum gradient methods called inertial Newton algorithm (INNA). Their works  demonstrate that the INNA implicitly utilizes the second-order information without evaluating the Hessian of $f$, which explains the high efficiency for these momentum accelerated gradient methods. 

However, with the presence of manifold constraint $x \in \M$ in \eqref{Prob_Rie}, how to develop efficient optimization approaches for \eqref{Prob_Rie} becomes challenging. Following the well-recognized framework proposed by \cite{absil2009optimization}, a great number of unconstrained optimization approaches can be extended to Riemannian manifolds, such as conjugate gradient method \cite{sato2016dai,sato2015new}, Newton method \cite{adler2002newton,hu2018adaptive}, quasi-Newton method \cite{huang2015broyden,kasai2018riemannian} and trust-region method \cite{absil2007trust,boumal2011rtrmc}.
Existing Riemannian momentum accelerated gradient methods are either developed based on the fast iterative shrinkage-thresholding algorithm \cite{huang2022riemannian}, or use the fixed momentum parameter, which follow the  updating scheme below:
\begin{equation}
	x_{k+1} = \Re_{x_k}(-\eta_{k,1}\g f(x_k) + \eta_{k,2}d_k),
\end{equation}
where $\eta_{k,1}$ and $\eta_{k,2}$ are stepsizes for Riemannian gradient  and the momentum, respectively. On the other hand, when the Riemannian Hessian of $f$ is available, one can develop the Riemannian trust-region method \cite{absil2009optimization} based on the unconstrained trust-region method, which iterates by sequentially solving the following trust-region subproblem,
\begin{equation}
	\begin{aligned}
		d_k ={}& \mathop{\arg\min}_{d \in {\rm T}_{x_k}\M, \| d \| \leq \Delta_k} f(x_k) + \<d, \g f(x_k) \> + \frac{1}{2}  \<d, \h f(x_k)[d] \>\\
		x_{k+1} ={}& \Re_{x_k}(d_k)
	\end{aligned}
\end{equation}
where ${\rm T}_{x_k}\M$ denotes the tangent of $\mathcal{M}$ at $x_k$ and 
$\mathcal{R}_{x_k}$ is a retraction mapping defined at $x_k$. 
Here $\Delta_k$ is the so-called trust-region radius that is adaptively updated in the Riemannian trust-region method. However, how to choose the trust-region radius in each iteration is challenging. To this end, Agarwal et al. \cite{agarwal2021adaptive} propose the Riemannian cubic regularization method, where they consider the following subproblem in each iteration:
\begin{equation}
	\begin{aligned}
		d_k ={}& \mathop{\arg\min}_{d \in {\rm T}_{x_k}\M, \| d \| \leq \Delta_k} f(x_k) + \<d, \g f(x_k) \> + \frac{1}{2}  \<d, \h f(x_k)[d] \> + \frac{\gamma_k}{6}\| d \|^3\\
		x_{k+1} ={}& \Re_{x_k}(d_k).
	\end{aligned}
\end{equation}
As illustrated in \cite{cartis2010complexity}, the iteration complexity of Riemannian cubic regularization method is $\mathcal{O}(\varepsilon^{-3/2})$. However, in these approaches, solving the subproblem is costly even in the {Euclidean} setting, as one need to intensively perform the Hessian-vector product, especially in high dimensional cases. 
To alleviate the high computational cost of the trust-region cubic regularization method in the
{Euclidean} setting, 
Zhang et al. \cite{zhang2022drsom} proposed a dimension-reduced second order method (DRSOM), where they employ a   two dimensional trust region subproblem along the current gradient and the momentum direction in each iteration. Therefore, their proposed DRSOM exploits the second order information while maintaining low computational cost to achieve 
highly efficient numerical performance in solving a wide variety of unconstrained optimization problems \cite{zhang2022drsom}. Very recently, this idea has  been extended to adaptive gradient method by Li et al. in \cite{li2022dimension}. 

In this work, we propose a Riemannian Dimension-Reduced Second Order Method (RDRSOM), which extends the DRSOM from the {Euclidean} setting to the Riemannian setting. In each iteration of our algorithm, we compute the next iterate by solving a cubic-regularized subproblem on a selected subspace of the tangent space at the current iterate. We prove that our algorithm returns an $\varepsilon$-stationary point of \cref{Prob_Rie} within $\O(1/\epsilon^{3/2})$ iterations under certain regularity conditions, which matches the existing results on the worst-case iteration complexity of Riemannian cubic regularized Newton methods \cite{agarwal2021adaptive}. Compared with these existing works that require solving the cubic regularized subproblem in the entire tangent subspace of the current iterate,  our proposed approach only partially exploits second-order information in a subspace of the tangent space, hence enjoys lower  per-iteration computational cost. Moreover, in order to further reduce computational complexity, we design a practical version of RDRSOM, where we employ the finite-difference techniques to approximate the Hessian in the particular subspace. Our proposed RDRSOM algorithm can also be viewed as an extension of the Barzilai-Borwein gradient method, which utilizes a finite difference scheme between the current and previous iterates to approximate the curvature information. We perform extensive numerical experiments to demonstrate that our proposed RDRSOM algorithm outperforms various
state-of-the-art optimization approaches in terms of computation time.

\subsection{Sensor network localization}\label{subsec-SNL}
One important application of our approach in the sensor network localization (SNL)  \cite{biswas2004semidefinite,so2007theory}, which computes the locations of a set of points in $\R^d$ according to their partial pairwise distance measurements. The SNL problem has extensive applicability across diverse fields, such as molecular conformation, dimensionality reduction, and ad hoc wireless sensor networks \cite{alfakih1999solving,biswas2004semidefinite,biswas2006distributed,doherty2001convex,hendrickson1995molecule,savvides2001dynamic,shang2003localization}. One mathematical formulation of this problem can be presented as follows:
\begin{multline}\label{SNL}
\min\Bigg\{ \frac{1}{2} \sum_{(i,j)\in \N}\( \|R_i-R_j\|^2-d_{ij}^2 \)^2-\frac{\lambda}{2n}\sum_{i=1}^n\sum_{j=1}^n \|R_i-R_j\|^2:\\
R\in \R^{n\times d},\ \forall\ (i,k)\in \G,\ \|R_i-a_k\|=d_{ik}\Bigg\},
\end{multline}
where $R_i$ denotes the $i$th row of $R$, $\N\subset \binom{[n]}{2}, \G\subset [n]\times [m],$ $a_k$'s are the positions of the anchors, $d_{ij}$ and $d_{ik}$ are pairwise sensor-sensor and sensor-anchor distance measurements, respectively. Our goal is to determine the locations of sensors $R_i$'s from the incomplete pairwise distance information. The regularization term  with $\lambda,$ which has been used before \cite{biswas2006semidefinite1,biswas2006semidefinite}, is added to prevent the predicted points from crowding together. We fix the distance between the sensors and the anchors by assuming that the sensor-anchor distance measurement is accurate while the sensor-sensor distance measurements may contain noise. This formulation has also been considered in \cite{li2018qsdpnal} by Li et al. One advantage of adding these constraints is that the feasible set is compact when all sensors are directly connected to some anchors.

Because of the non-convexity of problem \cref{SNL}, semidefinite programming relaxation
is widely used in the literature of (SNL) \cite{biswas2006semidefinite1,biswas2006semidefinite,biswas2004semidefinite,so2007theory}. The semidefinite relaxation of \cref{SNL} is the following quadratic SDP problem: \cite{li2018qsdpnal}
\begin{eqnarray}\label{SNLSDP}
&&\min\Bigg\{ \frac{1}{2}\sum_{(i,j)\in \N}\( g_{ij}^\top X g_{ij}-d_{ij}^2 \)^2
-\lambda\<\hat{I}-aa^\top/n,X \>
\\
&& \hspace{2cm} g_{ik}^\top X g_{ik}=d_{ik}^2\;\; \forall (i,j)\in \mathcal{G},\ 
 X=\begin{pmatrix}I_d&*\\ *&*\end{pmatrix}\in \S^{n+d}_+ \Bigg\}, \nonumber
\end{eqnarray}
where $\hat{I}:=\dd([0_{d};e]),$ $a: = [0_{d};e],$ $g_{ij}:=[0_d;e_i-e_j]$ and $g_{ik}:=[-a_k;e_i].$ Although the effectiveness of semidefinite programming has been demonstrated in numerous numerical experiments, this formulation is unscalable because of its large dimensionality
of  $\O(n^2).$ One way to reduce the dimension of \cref{SNLSDP} is to use 
its low rank decomposition, which is also known as Burer and Monteiro factorization \cite{burer2003nonlinear,burer2005local}, as described below: 
\begin{multline}\label{SNL1}
\min\Bigg\{ \frac{1}{2} \sum_{(i,j)\in \N}\( \|\hat{R}_i-\hat{R}_j\|^2-d_{ij}^2 \)^2-\frac{\lambda}{2n}\sum_{i=1}^n\sum_{j=1}^n \|\hat{R}_i-\hat{R}_j\|^2:\\
\hat{R}\in \R^{n\times r},\ \forall\ (i,k)\in \G,\ \|\hat{R}_i-\hat{a}_k\|=d_{ik}\Bigg\},
\end{multline}
where $r>d$ and $\hat{a}_k=[a_k,0_{1\times (r-d)}].$ Problem \cref{SNL1} is similar to \cref{SNL} and they are equivalent if $r=d.$ Despite the nonconvex nature of problem \cref{SNL1}, it has been shown by Boumal et al. that the non-convexity of low rank SDP problem is benign and one can use local search methods to 
find its global optimal solution provided that $r$ is larger than some rank bound \cite{Boumal1,boumal2016non}. 

In Section~\ref{sec-SNL}, we will apply (RDRSOM) to solve \cref{SNL} and \cref{SNL1}. In the literature of Riemannian optimization, the manifolds used in real applications are mostly well-known manifolds\footnote{Readers may read chapter two and seven in the book \cite{boumal2020introduction} to know several commonly used manifolds in Riemannian optimization.} such as Stiefel manifold ${\rm St}_{n,k}:=\big\{ X\in \R^{n\times k}:\ X^\top X=I_k \big\},$ oblique manifold ${\rm OB}_{n,r}:=\big\{ R\in \R^{n\times r}:\ \dd(RR^\top)=e \big\}$ and fixed-rank matrix manifold $\R_r^{m\times n}:=\left\{ X\in \R^{m\times n}:\ \rr(X)=r\right\}.$ These manifolds have simple well-studied geometric structures. As for problem \cref{SNL} and \cref{SNL1}, the constraints indicate that the distances between a sensor and several anchors are fixed. Thus, such a sensor lies on the intersection of several spheres with different radiuses and centers. While manifolds with spherical structure frequently appears in Riemannian optimization, as far as we know, the intersection of different spheres hasn't been considered in the literature before. Actually, it is not obvious whether the feasible set of \cref{SNL} is indeed a manifold. We will prove that the feasible set of \cref{SNL} is a Riemannian manifold even if the linear independence constraint qualification (LICQ) does not hold. This allows us to use our Riemannian optimization algorithm to solve \cref{SNL} efficiently.

\subsection{Organization}
The rest of this paper is organized as follows. In Subsection~\ref{subsec-notation}, we 
present some notations that are frequently used throughout the paper. In Section~\ref{sec-RDRSOM}, we develop a cubic-regularized Riemannian optimization method and conduct its convergence analysis. In Section~\ref{sec-SNL}, we prove that the feasible set of \cref{SNL} is a Riemannian manifold. In Section~\ref{sec-exp}, we perform numerical experiments to demonstrate the efficiency of our algorithm.

\subsection{Notations}\label{subsec-notation}
Through out this paper, $I_k$ denotes $k\times k$ identity matrix, $e$ denotes a vector of all ones. We omit the dimension if it is already clear from the context. For any matrix $A\in \R^{n\times m}$ and any $i\in [n],$ $A_i\in \R^{1\times m}$ denotes the $i$th row of $A.$ 
We use $\<A,B\>:={\rm tr}\(AB^\top\)$ to denote the matrix inner product and $\|\cdot\|$ to denote the Frobenius norm.

\section{The extension of DRSOM on manifold}\label{sec-RDRSOM}
\subsection{Preliminaries of Riemannian optimization}
We first recall some basic properties of Riemannian optimization\footnote{For text books on Riemannian optimization, please refer to \cite{absil2009optimization,boumal2020introduction}.}. For each point $x\in \M,$ there exists a tangent space $\T_x\M$ which can be viewed as a linearization of $\M$ at $x.$ If the manifold is given by $\M:=\{ x\in \E:\ c_i(x)=0,\ \forall i\in [m] \},$ where the LICQ property holds, then the tangent space is $\T_x\M:=\{h\in \E:\ \<\nabla c_i(x),h\>=0,\ \forall i\in [m] \}.$ The projection mapping $\P_x:\E\rightarrow \T_x\M$ is defined as the metric projection of any vector in $\E$ onto $\T_x\M.$ The Riemannian gradient $\g f(x)\in \T_x\M$ is defined as $\P_x \nabla f(x) .$ The Riemannian Hessian $\h f(x):\T_x\M\rightarrow \T_x\M$ is defined as 
$\h f(x) [h]:=\P_x \(\D \g f(x) [h]\)$ for any $h\in \T_x\M.$
It is the differential (pushforward) of the Riemannian gradient on $\T_x \M.$ For any $x\in \M,$ the retraction mapping $\Re_x:\T_x \M\rightarrow \M$ satisfies that $\Re_x(0)=x$ and ${\rm D}\Re_x(0)$ is the identity map. For any $x,y\in \M$ and $u\in \T_x\M,$ $\T_{y\leftarrow x}u$ is the vector transport which is the projection of $u\in \T_x\M$ on the tangent space $\T_y\M.$ 


\subsection{Riemannian DRSOM}

With the basic knowledge of Riemannian optimization, we are now able to state our algorithm
in Algorithm \ref{alg1}.
\begin{algorithm}
\caption{RDRSOM}
\begin{algorithmic}\label{alg1}
\STATE{{\bf Input:} $x_0\in \M,$ $\epsilon>0$}
\STATE{{\bf Initialization:} $k\gets 0$ }
\WHILE{$\|\g f(x_k)\|\geq \epsilon$}
\STATE{Choose a subspace $\V_k\subset \T_{x_k}\M$ and a regularization parameter $\gamma_k>0$}
\STATE{$d_k\gets \arg\min\left\{ f(x_k)+\<\g f(x_k),d\>+\frac{1}{2}\< \h f(x_k)[d],d \>+\frac{\gamma_k}{6}\| d \|^3:\ d\in \V_k \right\}$}
\STATE{$x_{k+1}\gets \Re_{x_k}(d_k)$}
\STATE{$k\gets k+1$}
\ENDWHILE
\end{algorithmic}
\end{algorithm}

In every iteration of Algorithm~\ref{alg1}, the next iterate is computed by solving the cubic regularized Newton subproblem restricted  in an adaptively selected subspace $\V_k$ of $\T_{x_k}\M$.  When we choose $\V_k=\T_x\M,$,  Algorithm~\ref{alg1} coincides with the  Riemannian cubic regularized Newton method. On the other hand, when we choose $\V_k$ as the two-dimensional subspace spanned by the directions of the gradient and the updating directions of the last iterate, i.e., 
$$\V_k:={\rm span}\(\g f(x_k), \P_{x_k}(x_k-x_{k-1})\),$$
 then Algorithm~\ref{alg1} becomes the extension of DRSOM on manifold with the difference that we use cubic regularization to control the step size instead of trust region. Before we conduct the convergence analysis of Algorithm~\ref{alg1}, we need the following assumption.

\begin{assumption}\label{LipSec}
	For Algorithm \ref{alg1}, we assume that there exists $M>0$ such that the following conditions hold, 
\begin{multline}\label{AS1}
\left|f(\Re_{x_k}(d_k))-\left(f(x_k)+\<d_k,\g f(x_k)\>+\frac{1}{2}\<d_k,\h f(x_k)[d_k]\>\right)\right|\\
\leq \frac{M}{6}\|d_k\|^3.
\end{multline}
\begin{equation}\label{AS2}
\left\| \g f(\Re_{x_k}(d_k))-\g f(x_k)-\h f(x_k)[d_k] \right\|\leq \frac{M}{2}\|d_k\|^2.
\end{equation}
\begin{equation}\label{AS3}
\left\| \P_{\V_k}\( \h f(x_k)[d_k]\) -\h f(x_k)[d_k] \right\|\leq \frac{M}{2}\|d_k\|^2,
\end{equation}
where $\P_{\V_k}(\cdot)$ is the orthogonal projection mapping of the linear space $\V_k.$
\end{assumption}
In Assumption~\ref{LipSec}, \cref{AS1} and \cref{AS2} are common assumptions in existing works on  the convergence properties of Riemannian trust region method \cite{boumal2019global}. Assumption~\eqref{AS3} imposes a regularity condition on $\P_{\V_k}\( \h f(x_k)\)$ in the sense that it approximates $\h f(x_k)$ along the direction $d_k$. Such a condition also appears in \cite{zhang2022drsom}. While \eqref{AS3} is a standard assumption in cubic regularized Newton methods \cite{agarwal2021adaptive,cartis2010complexity}, it could be restrictive when $\V_k$ is a lower dimensional space. We will discuss how to avoid this assumption in the next section.

The following lemma shows the relationship between $\| \g f(x_{k+1}) \|$ and $\| d_k \|^2$ under Assumption \ref{LipSec}. 

\begin{lemma}\label{lemgrad}
Suppose Assumption~\ref{LipSec} holds and $\g f(x_k)\in \V_k$, then 
\begin{equation}
\| \g f(x_{k+1}) \| \leq \frac{2M + \gamma_k }{2}  \| d_k \|^2. 
\end{equation}
\end{lemma}

\begin{proof}
The first order optimality condition of the subproblem of Algorithm~\ref{alg1} implies that
\begin{equation}\label{Eq_Lem1_xnc}
\g f(x_k)+\P_{\V_k}\( \h f(x_k)[d_k] \)+\frac{\gamma_k\|d_k\|}{2}d_k=0,
\end{equation}
which implies that 
\begin{align}
\| \g f(x_k)+\h f(x_k)[d_k] \|
\leq{}& \| \g f(x_k)+\P_{\V_k}\( \h f(x_k)[d_k] \) \| \notag\\
&+\| \h f(x_k)[d_k]-\P_{\V_k}\( \h f(x_k)[d_k] \) \|\notag\\
\leq{}& \frac{\gamma_k+M}{2}\|d_k\|^2,\label{Jan_19_1}
\end{align}
where the second inequality comes from  \cref{AS3} and \cref{Eq_Lem1_xnc}. Combining \cref{AS2} and \cref{Jan_19_1}, we get the required result.
\end{proof}

\begin{lemma}\label{lemprod}
Suppose Assumption~\ref{LipSec} holds, $\g f(x_k)\in \V_k$ and $\gamma_k\geq M,$ then it holds that 
\begin{equation}
\< \g f(x_k),d_k \>\leq 0.
\end{equation}
\end{lemma}

\begin{proof}
From the second order optimality of the subproblem of Algorithm~\ref{alg1}, we have that
\begin{equation}\label{Jan_19_3}
\< \h f(x_k)[d_k],d_k \>+\frac{\gamma_k }{2}\|d_k\|^3\geq 0.
\end{equation}
From \cref{Eq_Lem1_xnc}, we have that
\begin{equation}\label{Jan_19_4}
\<\g f(x_k),d_k\>+\< \h f(x_k)[d_k],d_k \>+\frac{\gamma_k }{2}\|d_k\|^3=0.
\end{equation}
Combining \cref{Jan_19_3} and \cref{Jan_19_4}, we get $\< \g f(x_k),d_k \>\leq 0.$
\end{proof}

Now we state our main convergence theorem.

\begin{theorem}\label{thm-conv}
Suppose Assumption~\ref{LipSec} holds, and for any $k\in \mathbb{N},$ $\g f(x_k)\in \V_k$ and $C>\gamma_k\geq M.$ Moreover, suppose that the optimal value of problem \cref{Prob_Rie} is $f^*>-\infty.$ Then the sequence generated by Algorithm~\ref{alg1} satisfies
\begin{equation}\label{eq:conv}
\min_{0 \leq k \leq N+1} \| \g f(x_{k})  \| \leq \frac{12^{2/3}(2M+C)}{2M^{2/3}} \left( \frac{f(x_0) - f^*}{N} \right)^{2/3}. 
\end{equation}
Consequently, Algorithm \ref{alg1} will terminate within $\O(1/\epsilon^{3/2})$ iterations.
\end{theorem}

\begin{proof}
Let $m_k(d):=f(x_k)+\<\g f(x_k),d\>+\frac{1}{2}\< \h f(x_k)[d],d \>.$ Then it holds that
\begin{equation}\label{Jan_19_5}
f(x_k) - m_k(d_k) = -\<\g f(x_k), d_k  \> - \frac{1}{2} \<\h f(x_k)[d_k], d_k  \> .
\end{equation}
Combining \cref{Jan_19_4} and \cref{Jan_19_5} we obtain that
\begin{equation}\label{Jan_19_6}
f(x_k) - m_k(d_k) = -\frac{1}{2}\<\g f(x_k), d_k  \> +\frac{\gamma_k}{4} \| d_k \|^3\geq \frac{\gamma_k}{4} \| d_k \|^3,
\end{equation}
where the inequality comes from Lemma~\ref{lemprod}. From \cref{AS1}, \cref{Jan_19_6} and $\gamma_k\geq M$, we get
\begin{equation}\label{Jan_19_7}
f(x_k) - f(x_{k+1}) \geq f(x_k)-m_k(d_k)-|f(x_{k+1})-m_k(d_k)|\geq \frac{M}{12}\|d_k\|^3.
\end{equation}
From Lemma~\ref{lemgrad} and \cref{Jan_19_7}, we immediately get
\begin{equation}\label{Jan_19_8}
f(x_k) - f(x_{k+1}) \geq \frac{2^{3/2}M}{12(2M+C)^{3/2}}\| \g f(x_{k+1})\|^{3/2}.
\end{equation}
Therefore, we get
\begin{equation}\label{Jan_19_9}
f(x_0) - f^* \geq \sum_{k=0}^N\frac{2^{3/2}M}{12(2M+C)^{3/2}}\| \g f(x_{k+1})\|^{3/2}.
\end{equation}
From here, the required result follows.
\end{proof}

%


\subsection{Practical RDRSOM}
The convergence of Algorithm~\ref{alg1} is based on Assumption~\ref{LipSec}. Although conditions \eqref{AS1} and \eqref{AS2} are relatively mild in practice and commonly appear in \cite{cartis2010complexity,zhang2022drsom},  condition \ref{AS3} may not hold if the subspace $\V_k$ is not large enough. To satisfy condition \eqref{AS3}, we can progressively increase the dimension of $\V_k$. However, in practice, to maintain low computational cost, it is usually better to choose a smaller dimensional subspace $\V_k$ of $\E$  by trading off the iteration complexity. Let $V_k$ be a matrix with the columns forming an orthogonal basis of $\V_k$. Then the subproblem can then be reformulated as follows:
\begin{multline}\label{rprob}
\min\Big\{ f(x_k)+ \< V_k^\top \g f(x_k),h \>+\frac{1}{2}\< V_k^\top \h f(x_k)[V_k h],h \>+\frac{\gamma_k}{6}\|h\|^3:\\ h\in \R^{{\rm dim}\V_k} \Big\}.
\end{multline}
Since the size of problem \cref{rprob} is small, it can be efficiently solved by an iterative solver such as the trust-region method or the method described in Section 5.1 of \cite{nesterov2006cubic}. The main cost of solving problem \cref{rprob} is  in computing the dimension-reduced Hessian $Q_k:=V_k^\top \h f(x_k)[V_k \cdot]$, which only requires matrix-vector multiplications.  Let $h_k^i:=\P_{x_k}\(x_k-x_{k-i}\)$ and $g_k:=\g f(x_k).$ When we choose $\V_k$ as
\begin{equation}\label{eq:subspace}
\V_k^r:={\rm span}\( h_k^1,h_k^2,\ldots,h_k^r,g_k \),
\end{equation}
 we can use the finite differences
\begin{equation}\label{eq-app}
 \h f(x_k)[h_k^i]\approx g_k-g_{k-i},\ \h f(x_k) [g_k]\approx\frac{ \g f(\Re_{x_k}(\eta g_k))-g_k}{\eta\|g_k\|}
 \end{equation}
 to approximate the Hessian-vector multiplications. Here $\eta>0$ is the parameter for the 
 last finite-difference.   It is worth mentioning  that the finite difference scheme for approximating $\h f(x_k) [g_k]$ in \cref{eq-app} requires additional computational costs for computing the Riemannian gradient of $f$ at $\Re_{x_k}(\eta g_k)$. Hence there is a trade-off between the per-iteration computational cost and the overall number of iterations for RDRSOM. Furthermore, since choosing $\V_k$ by \eqref{eq-app} may not guarantee that condition~\eqref{AS3} holds, we prove a weaker complexity bound of $\O(1/\epsilon^2)$ in the following theorem, which is sharp for the gradient descent and Newton method \cite{boumal2019global,cartis2010complexity}.

\begin{theorem}\label{conv2}
Suppose for any $k\in \mathbb{N},$ $\gamma\geq \gamma_k\geq M,$ $M I \succeq V_k^\top \h f(x_k)[V_k\cdot],$ $\g f(x_k)\in \V_k$ and condition \eqref{AS1} holds. Moreover, suppose that the optimal value of problem \cref{Prob_Rie} is $f^*>-\infty.$ Then the sequence generated by Algorithm~\ref{alg1} satisfies that

\begin{equation}\label{eq:conv2}
\min_{0\leq k\leq N}\| \g f(x_k)\| =\O(1/\sqrt{N}).
\end{equation}
Then for any $\epsilon>0,$ Algorithm~\ref{alg1} will terminate within $\O(1/\epsilon^{2})$ iterations.
\end{theorem}

\begin{proof}
From \cref{AS1}, we have that
\begin{align}
&f(x_{k+1})\leq f(x_k)+\<\g f(x_k),d_k\>+\frac{1}{2}\< \h f(x_k)[d_k],d_k \>+\frac{M}{6}\| d _k\|^3\notag \\
&\leq f(x_k)+\<\g f(x_k),d_k\>+\frac{1}{2}\< \h f(x_k)[d_k],d_k \>+\frac{\gamma_k}{6}\| d _k\|^3\notag\\
&= \min\big\{ f(x_k)+\<\g f(x_k),d\>+\frac{1}{2}\< \h f(x_k)[d],d \>+\frac{\gamma_k}{6}\| d \|^3: d\in \V_k \big\}\notag \\
&\leq \inf_{\eta\geq 0}\left\{f(x_k)-\eta\|\g f(x_k)\|^2+\frac{M\eta^2}{2}\| \g f(x_k) \|^2+\frac{\eta^3 \gamma}{6}\| \g f(x_k)\|^3\right\},\label{Jan_24_2}
\end{align}
where the first inequality comes from \cref{AS1}, the second inequality comes from $\gamma_k\geq M,$ the last inequality comes from $M I \succeq V_k^\top \h f(x_k)[V_k\cdot],$ $\gamma\geq \gamma_k$ and $\g f(x_k)\in \V_k.$ Choose 
$$\eta=\min\left\{ \frac{1}{2M},\frac{3M}{\gamma \| \g f(x_k)\|} \right\}$$
 in \cref{Jan_24_2}. Because $\eta\leq 3M/(\gamma \| \g f(x_k)\|),$ we have that 
 \begin{equation}\label{Jan_24_3}
 \frac{\eta^3 \gamma}{6}\| \g f(x_k)\|^3\leq \frac{M\eta^2}{2}\| \g f(x_k) \|^2.
\end{equation}
By our choice of $\eta$, we get $-\eta + M \eta^2 = -\frac{1}{4M}$ if $\eta = \frac{1}{2M}$, and $-\eta + M \eta^2 \leq - \frac{3M}{2\gamma \| \g f(x_k) \|}$ if $\eta = \frac{3M}{\gamma \| \g f(x_k) \|}$.
Thus
we have that
 \begin{equation}\label{Jan_24_4}
\(-\eta+M\eta^2\)\| \g f(x_k) \|^2\leq-\min\left\{ \frac{1}{4M}, \frac{3M}{2\gamma \| \g f(x_k)\|}\right\}\| \g f(x_k) \|^2.
\end{equation}
Substituting \cref{Jan_24_3} and \cref{Jan_24_4} into \cref{Jan_24_2}, we get 
\begin{equation}\label{Jan_24_5}
f(x_{k+1})\leq f(x_k)-\min\left\{ \frac{\| \g f(x_k) \|^2}{4M}, \frac{3M\| \g f(x_k) \|}{2\gamma }\right\}.
\end{equation}
Taking summation of \cref{Jan_24_5} from $0$ to $k$, we get
\begin{equation}\label{Jan_25_1}
\min_{0\leq k\leq N} \min\left\{ \frac{\| \g f(x_k) \|^2}{4M}, \frac{3M\| \g f(x_k) \|}{2\gamma }\right\}\leq \frac{f(x_0)-f^*}{N}.
\end{equation}
Therefore, we have that
\begin{equation}\label{Jan_25_2}
\begin{aligned}
	\min_{0\leq k\leq N}\| \g f(x_k) \|\leq{}& \max\left\{\sqrt{\frac{4M(f(x_0)-f^*)}{\beta N}},\frac{2\gamma (f(x_0)-f^*)}{3MN}\right\}\\={}&\O(1/\sqrt{N}).
\end{aligned}
\end{equation}
Thus, Algorithm~\ref{alg1} will terminate within $\O(1/\epsilon^{2})$ iterations.
\end{proof}

\section{Feasible set of sensor network localization}\label{sec-SNL}

In this section, we aim to prove that the feasible set of \cref{SNL} is a Riemannian manifold. Since the constraints of different rows of $R$ are independent, we only need to prove that the feasible set of every row of $R$, i.e., the intersection of different spheres, is a Riemannian manifold. Given nonnegative integers $k,r$, let $\left\{ y_i:\ i\in [k] \right\}\subset \R^r$ and $d\in \R^k_{+}$. We define the matrix $R:=[y_1,y_2,\ldots,y_k]^\top\in \R^{k\times r}$. Moreover, we define the following set:
\begin{equation}\label{Intp}
\B_{R,d}:=\left\{ x\in \R^r:\ \|x-y_i\|^2=d_i^2,\ \forall\ i\in [k]\right\},
\end{equation}
which can be regarded as the intersection of several spheres with centers $y_i$'s and radii $d_i$'s. The following proposition shows that $\B_{R,d}$ is a Riemannian manifold.

\begin{proposition}\label{Intpmani}
Suppose $\B_{R,d}\neq \emptyset.$ Then the following two statements hold:
\begin{itemize}
\item[(i)] If $\rr\( [e,R]\)=k$ and there exists $x\in \B_{R,d}$ without LICQ, then $\B_{R,d}=\{x\}.$
\item[(ii)] If $\rr\( [e,R]\)\leq k-1,$ then there exists $\I \subset [k]$ such that $\B_{R,d}=\B_{R_\I,d_\I}.$ 
\end{itemize}
Moreover, $\B_{R,d}$ is Riemannian submanifold in $\R^r.$
\end{proposition}
\begin{proof}
We first prove (i).\\
Assume on the contrary that $|\B_{R,d}|>1$. Because $x$ doesn't satisfy LICQ, there exists $\lambda\in \R^k$ such that $\lambda\neq 0$ and 
\begin{equation}\label{July_23_1}
 \lambda^\top \(ex^\top-R\)=0.
\end{equation}
 Because $|\B_{R,d}|>1,$ there exists $z\in \B_{R,d}$ such that $z\neq x.$ Define $h:=z-x\neq 0.$ For any $i\in [k],$ define the function $f_i:\R\rightarrow \R$ such that $f_i(t):=\| x+th- y_i\|^2-d_i^2.$ Because $x,z\in \B_{R,d},$ we have that $f_i(0)=f_i(1)=0.$ This implies that $\|h\|^2+2\<h,x-y_i\>=0$ for any $i\in [k].$ Thus, we get the following equation
\begin{equation}\label{July_23_2}
\(ex^\top-R\)h=-\|h\|^2e/2.
\end{equation}
From \cref{July_23_1} and \cref{July_23_2}, we have that 
\begin{equation}\label{July_23_3}
\lambda^\top e=-2\lambda^\top \(ex^\top-R\)h/\|h\|^2=0.
\end{equation}
Substituting $\lambda^\top e=0$ into (\ref{July_23_1}), we get $\lambda^\top R=0.$ Therefore, we have that $\lambda^\top [e,R]=0,$ which contradicts to the fact that $\rr\(  [e,R] \)=k.$ \\

Now, we move on to prove (ii). Because $\rr\([e,R]\)<k$, we have that there exists $\lambda\in \R^k$ such that $\lambda\neq 0$ and $\lambda^\top [e,R]=0.$ This implies that 
 \begin{equation}\label{July_23_4}
 \forall x\in \R^r,\ \sum_{i=1}^k \lambda_i \|x\|^2\;=\; 0 \;=\;
  \sum_{i=1}^k \lambda_i \<x,y_i\>.
 \end{equation}
 Suppose $z\in \B_{R,d},$ then we have that 
\begin{align}
&0=\sum_{i=1}^k \lambda_i\(\| z-y_i\|^2-d_i^2\)=\sum_{i=1}^k \lambda_i\(\|z\|^2-2\<z,y_i\>+\|y_i\|^2-d_i^2\)\notag \\
&=\sum_{i=1}^k \lambda_i \|z\|^2-2\sum_{i=1}^k \lambda_i \<z,y_i\>+\sum_{i=1}^k\lambda_i\( \|y_i\|^2-d_i^2 \)=\sum_{i=1}^k\lambda_i\( \|y_i\|^2-d_i^2 \),\label{July_23_5}
\end{align}
where the last equality comes from (\ref{July_23_4}). Because $\lambda\neq 0$, there exists $j\in [k]$ such that $\lambda_j\neq 0.$ From (\ref{July_23_4}) and (\ref{July_23_5}), we have that for any $x\in \R^r,$
\begin{align}
&\lambda_j\( \|x-y_j\|^2-d_j^2\)=\lambda_j\( \|x\|^2-2\<x,y_j\>+\|y_j\|^2-d_j^2\)\notag \\
&=-\sum_{i\in [k]\setminus\{j\}} \lambda_i \|x\|^2+2\sum_{i\in [k]\setminus\{j\}} \lambda_i\<x,y_i\>-\sum_{i\in [k]\setminus\{j\}} \lambda_i\( \|y_i\|^2-d_i^2 \)\notag \\
&=-\sum_{i\in [k]\setminus\{j\}}\lambda_i\(\|x-y_i\|^2-d_i^2\).\label{July_23_6}
\end{align}
From (\ref{July_23_6}) and $\lambda_j\neq 0,$ we know that as long as for any $i\in [k]\setminus\{j\},$ $\|x-y_i\|^2-d_i^2=0,$ then $\|x-y_j\|^2-d_j^2=0.$ This implies that $\|x-y_j\|^2=d_j^2$ is a redundant constraint and $\B_{R,d}=\B_{R_{[k]\setminus\{j\}},d_{[k]\setminus\{j\}}}$.  

Finally, we prove that $\B_{R,d}$ is a Riemannian manifold.  if $\rr\( [e,R]\)=k,$ then from (i), $\B_{R,d}$ is a singleton or the LICQ holds everywhere on $\B_{R,d}.$ For each case, $\B_{R,d}$ is a Riemannian submanifold in $\R^r.$ If $\rr\( [e,R]\)<k,$ then from (ii), there exists some redundant constraints in $\B_{R,d}.$ Because there are finitely many constraints in $\B_{R,d},$ after removing finitely many redundant constraints, we can reduce to the first case. Therefore, $\B_{R,d}$ is always a Riemannian submanifold in $\R^r$.
\end{proof}

\section{Numerical experiments}\label{sec-exp}
In this section,  we test the numerical performance of our proposed Algorithm \ref{alg1} against some state-of-the-art optimization approaches to illustrate the high efficiency of our 
 algorithm. All the experiments are conducted using Matlab R2021b on a Workstation with a Intel(R) Xeon(R) CPU E5-2680 v3 @ 2.50GHz Processor and 128GB RAM. In Algorithm \ref{alg1}, we choose the subspace $\V_k$ as 
 $$\V_k:={\rm span}\( \g f(x_k), \P_{x_k}\(x_k-x_{k-1}\) \)$$ 
 for any $k\geq 0$, and we employ the finite difference approximation scheme in \cref{eq-app} to compute an approximation for the dimension reduced Hessian of $f$ in $\V_k$. The parameter $\eta$ in \cref{eq-app} is set as $\eta = 10^{-7}$ for all the test instances. In our numerical experiments, we compare the performance of our proposed algorithm with several well-recognized state-of-the-art Riemannian optimization algorithms from the Manopt package \cite{boumal2014manopt}, including the trust region method (TR),  conjugate gradient method (CG), Broyden–Fletcher–Goldfarb–Shanno quasi-Newton method (BFGS) and Barzilai-Borwein method (BB). All the parameters are set as their default values for these approaches. The test problems are chosen to be low-rank max-cut SDP problems, 
 discretized problems of the Kohn-Sham equation, K-measn clustering and sensor network localization problems.  It is worth mentioning that the manifold in \cref{SNL} is not supported in Manopt, hence we compare our proposed algorithm with the build-in sequential quadratic programming (SQP)  and 
 interior point method (IPM) implemented in the {\sc Matlab} function
 {\tt fmincon} for solving \eqref{SNL}.

\subsection{Low-rank max-cut SDP problem}
In this subsection, we consider the following low rank formulation of the max-cut SDP problem \cite{burer2003nonlinear,burer2005local}, 
\begin{equation}\label{maxcut}
\min\left\{-\<L,RR^\top\>:\ \dd(RR^\top)=e,\ R\in \R^{n\times r}\right\},
\end{equation}
where $L$ is the Laplacian matrix of an undirected graph with $n$ vertices and $m$ edges. This problem has been used extensively as a benchmark problem for testing
Riemannian optimization methods and various methods implemented in Manopt 
have achieved impressive numerical efficiency in solving large problems with
$n$ going beyond $5000$.

In \eqref{maxcut}, we choose $r=\lceil \sqrt{2n}\rceil$ to be the theoretical rank bound given by \cite{pataki1998rank} to guarantee that \cref{maxcut} is equivalent to its convex SDP formulation. We terminate the algorithms if the norm of the Riemannian gradient is smaller than $10^{-4}.$ We set the 
maximum number of iterations of the trust region method to be $1000$ and all the other methods to be $10000$. We use the following KKT residue to measure the accuracy of the output:
\begin{multline}\label{KKT}
{\rm Resp}:=\frac{\|\dd(RR^\top)-e\|}{1+\sqrt{n}},\ {\rm Resd}:=\frac{\|\Pi_{\S_-^n}\( L-\dd(\lambda) \)\|}{1+\|L\|},\\
 {\rm Pdgap}:=\frac{|\<LR,R\>-\lambda^\top e|}{1+|\<LR,R\>|+|\lambda^\top e|},\ {\rm Residue}:=\max\{{\rm Resp}, {\rm Resd}, {\rm Pdgap}\},
\end{multline}
where the dual variable $\lambda\in \R^n$ can be recovered from the linear system (13) in \cite{boumal2016non}. We consider the Gset graphs\footnote{Dataset from \href{https://web.stanford.edu/~yyye/yyye/Gset/}{https://web.stanford.edu/~yyye/yyye/Gset/}.}, which are frequently used as benchmark test for max-cut problems. Since there are too many graphs in Gset, we only choose graphs whose number of vertices are at least $5000$.

\begin{center}
\begin{footnotesize}
\begin{longtable}{lclclcccclll} 
\caption{Comparison of  RDRSOM, TR, CG, BFGS, BB methods for low rank max-cut SDP.}\\
\hline
problem & metric & RDRSOM & TR & CG& BFGS & BB  \\ \hline
 g55 & Fval & -44157.8 & -44157.8 & -44157.8 & -44157.8 & -44157.8\\
 n=5000 & Residue & 4.6e-10 & 5.7e-11 & 2.3e-10 & 1.9e-09 & 3.6e-11 \\
 m=12498 & Time [s] &  9.3 & 11.9 & 16.1 & 41.4 & 36.4\\
 \hline
 g56 & Fval & -19040.0 & -19040.0 & -19040.0 & -19040.0 & -19040.0\\
 n=5000 & Residue & 2.8e-09 & 6.2e-11 & 9.6e-10 & 5.6e-09 & 6.5e-11 \\
 m=12498 & Time [s] &  6.9 &  9.4 & 12.6 & 34.1 & 19.5\\
 \hline
 g57 & Fval & -15542.0 & -15542.0 & -15542.0 & -15542.0 & -15542.0\\
n=5000 & Residue & 2.2e-10 & 2.3e-10 & 1.4e-09 & 3.5e-10 & 5.5e-09 \\
m=10000 & Time [s] & 32.5 & 138.0 & 81.4 & 209.8 & 135.4\\
\hline
g58 & Fval & -80544.8 & -80544.8 & -80544.8 & -80544.8 & -80544.8\\
n=5000 & Residue & 8.0e-11 & 4.7e-12 & 1.2e-10 & 3.2e-10 & 7.1e-12 \\
m=29570 & Time [s] & 28.4 & 34.1 & 46.2 & 128.6 & 106.2\\
\hline
g59 & Fval & -29249.3 & -29249.3 & -29249.3 & -29249.3 & -29249.3\\
n=5000 & Residue & 3.0e-10 & 2.2e-11 & 8.9e-10 & 1.9e-09 & 1.1e-09 \\
m=29570 & Time [s] & 21.4 & 23.6 & 39.7 & 99.4 & 82.0\\
\hline
g60 & Fval & -60889.1 & -60889.1 & -60889.1 & -60889.1 & -60889.1\\
n=7000 & Residue & 4.3e-10 & 3.0e-10 & 8.5e-10 & 6.3e-10 & 5.0e-10 \\
m=17148 & Time [s] & 14.4 & 19.3 & 19.7 & 63.4 & 37.1\\
\hline
g61 & Fval & -27312.4 & -27312.4 & -27312.4 & -27312.4 & -27312.4\\
n=7000 & Residue & 7.1e-10 & 8.5e-10 & 1.1e-09 & 1.7e-09 & 1.1e-09 \\
m=17148 & Time [s] & 21.0 & 90.9 & 32.0 & 92.1 & 61.4\\
\hline
g62 & Fval & -21723.6 & -21723.6 & -21723.6 & -21723.6 & -21723.6\\
n=7000 & Residue & 1.8e-10 & 8.0e-10 & 2.5e-09 & 5.1e-10 & 8.2e-09 \\
m=14000 & Time [s] & 71.6 & 373.0 & 123.3 & 521.6 & 210.8\\
\hline
g63 & Fval & -112977.7 & -112977.7 & -112977.7 & -112977.7 & -112977.7\\
 n=7000 & Residue & 8.3e-11 & 2.0e-11 & 2.8e-10 & 1.3e-10 & 3.5e-10 \\
 m=41459 & Time [s] & 79.6 & 582.7 & 114.4 & 378.1 & 324.4\\
 \hline
 g64 & Fval & -41863.6 & -41863.6 & -41863.6 & -41863.6 & -41863.6\\
 n=7000 & Residue & 5.0e-10 & 8.5e-11 & 6.1e-10 & 1.2e-09 & 9.1e-10 \\
 m=41459 & Time [s] & 66.2 & 72.7 & 134.1 & 297.1 & 327.2\\
 \hline
 g65 & Fval & -24822.2 & -24822.2 & -24822.2 & -24822.2 & -24822.2\\
 n=8000 & Residue & 1.1e-10 & 1.8e-10 & 1.6e-09 & 5.8e-10 & 8.4e-09 \\
 m=16000 & Time [s] & 74.6 & 740.0 & 146.1 & 544.1 & 254.5\\
 \hline
 g66 & Fval & -28308.9 & -28308.9 & -28308.9 & -28308.9 & -28308.9\\
 n=9000 & Residue & 5.5e-10 & 7.0e-10 & 1.2e-09 & 6.2e-10 & 6.8e-09 \\
 m=18000 & Time [s] & 100.0 & 501.8 & 163.1 & 810.3 & 307.8\\
 \hline
 g67 & Fval & -30977.7 & -30977.7 & -30977.7 & -30977.7 & -30977.7\\
 n=10000 & Residue & 1.3e-10 & 2.4e-10 & 9.7e-10 & 2.6e-10 & 8.3e-09 \\
 m=20000 & Time [s] & 127.9 & 1371.4 & 177.8 & 1114.4 & 356.9\\
 \hline
 g70 & Fval & -39446.1 & -39446.1 & -39446.1 & -39446.1 & -39446.1\\
 n=10000 & Residue & 2.2e-10 & 3.7e-12 & 1.6e-09 & 2.3e-10 & 3.4e-09 \\
 m=9999 & Time [s] & 36.8 & 288.4 & 63.5 & 250.8 & 100.7\\
 \hline
 g72 & Fval & -31234.2 & -31234.2 & -31234.2 & -31234.2 & -31234.2\\
 n=10000 & Residue & 8.5e-11 & 1.8e-12 & 5.8e-10 & 2.0e-10 & 1.1e-08 \\
 m=20000 & Time [s] & 112.8 & 881.2 & 191.9 & 907.5 & 359.2\\
 \hline
 g77 & Fval & -44182.7 & -44182.7 & -44182.7 & -44182.7 & -44182.7\\
 n=14000 & Residue & 7.6e-11 & 1.4e-10 & 7.1e-10 & 1.2e-10 & 1.0e-08 \\
 m=28000 & Time [s] & 272.6 & 1576.9 & 450.4 & 2402.6 & 603.8\\
 \hline
 g81 & Fval & -62624.8 & -62624.8 & -62624.8 & -62624.8 & -62624.8\\
 n=20000 & Residue & 4.6e-11 & 1.3e-10 & 1.4e-09 & 7.9e-11 & 2.0e-08 \\
 m=40000 & Time [s] & 662.7 & 4283.9 & 1219.0 & 6087.4 & 1062.1\\
 \hline
 \end{longtable}
 \end{footnotesize}
\end{center}

From the numerical results in Table 1, we can see that RDRSOM is faster than all the other algorithms implemented in Manopt in every instance. We should emphasize that the various algorithms in Manopt are well tested for
max-cut problems. Thus it is quite surprising that RDRSOM can  
perform even better than those well tested algorithms.
Among the algorithms in Manopt, the 
conjugate gradient method also behaves very well and its speed is close to RDRSOM for some instances. One reason is that the subspace $\V_k={\rm span}\( \g f(x_k), \P_{x_k}\(x_k-x_{k-1}\)\),$ contains the conjugate gradient direction. Also, it has been proved in \cite{zhang2022drsom} that DRSOM is exactly the conjugate gradient method for convex quadratic programming.

\subsection{Discretized 1D Kohn-Sham Equation}
In this subsection, we consider the following discretized problem of the 1D Kohn-Sham Equation problem \cite{lin2013elliptic,liu2014convergence}:
\begin{multline}\label{1KS}
\min\Big\{\frac{1}{2}{\rm tr}( R^\top LR )+\frac{\alpha}{4}{\rm diag}(RR^\top)^\top L^{-1} {\rm diag}(RR^\top):\\ R^\top R=I_p,\ R\in \R^{n\times r}\Big\},
\end{multline}
where $L$ is a tri-diagonal matrix with $2$ on its diagonal and $-1$ on its subdiagonal and $\alpha>0$ is a parameter. We choose $\alpha=1,$ $n\in \{1000,2000,5000,7000,10000\},$ $r\in\{20,50\}.$ We terminate the algorithms when the norm of the Riemannian gradient is smaller than $10^{-4}.$ We set the maximum number of iterations of the trust region method to be 1000 and all the other methods to be 10000.

\begin{center}
\begin{footnotesize}
\begin{longtable}{lclclcccclll} 
\caption{Comparison of  RDRSOM, TR, CG, BFGS, BB methods for discretized problems of
the 1D Kohn-Sham equation.}\\
\hline
problem & metric & RDRSOM & TR & CG& BFGS & BB  \\ \hline
n=1000 & Fval &  210.7 &  210.7 &  210.7 &  210.7 &  210.7\\
 r=20 & Gradnorm & 8.7e-05 & 4.1e-05 & 9.4e-05 & 8.7e-05 & 8.3e-05 \\
  & Time [s] &  0.2 &  2.5 &  1.0 & 13.3 &  0.9\\
 \hline
 n=1000 & Fval & 2810.7 & 2810.7 & 2810.7 & 2810.7 & 2810.7\\
 r=50 & Gradnorm & 9.9e-05 & 4.9e-08 & 9.9e-05 & 8.9e-05 & 9.9e-05 \\
  & Time [s] &  0.6 &  6.0 &  2.8 & 37.5 &  4.1\\
 \hline
n=2000 & Fval &  210.7 &  210.7 &  210.7 &  210.7 &  210.7\\
 r=20 & Gradnorm & 9.3e-05 & 4.7e-06 & 8.2e-05 & 8.6e-05 & 6.2e-05 \\
  & Time [s] &  0.2 &  4.4 &  1.3 & 15.6 &  1.1\\
 \hline
 n=2000 & Fval & 2810.7 & 2810.7 & 2810.7 & 2810.7 & 2810.7\\
 r=50 & Gradnorm & 1.0e-04 & 1.7e-05 & 9.7e-05 & 8.3e-05 & 1.0e-04 \\
  & Time [s] &  0.9 &  4.8 &  3.8 & 39.2 &  5.8\\
 \hline
 n=5000 & Fval &  210.7 &  210.7 &  210.7 &  210.7 &  210.7\\
 r=20 & Gradnorm & 9.2e-05 & 4.8e-05 & 9.3e-05 & 7.6e-05 & 6.2e-05 \\
  & Time [s] &  0.3 &  2.0 &  1.8 & 15.4 &  1.8\\
 \hline
 n=5000 & Fval & 2810.7 & 2810.7 & 2810.7 & 2810.7 & 2810.7\\
 r=50 & Gradnorm & 9.8e-05 & 2.2e-07 & 8.2e-05 & 8.0e-05 & 9.2e-05 \\
  & Time [s] &  1.7 & 21.1 &  5.6 & 42.9 &  7.1\\
 \hline
 n=7000 & Fval &  210.7 &  210.7 &  210.7 &  210.7 &  210.7\\
 r=20 & Gradnorm & 9.9e-05 & 1.1e-06 & 7.6e-05 & 9.1e-05 & 9.8e-05 \\
  & Time [s] &  0.5 &  3.9 &  1.8 & 17.2 &  2.4\\
 \hline
 n=7000 & Fval & 2810.7 & 2810.7 & 2810.7 & 2810.7 & 2810.7\\
 r=50 & Gradnorm & 9.1e-05 & 1.3e-05 & 9.4e-05 & 9.6e-05 & 8.8e-05 \\
  & Time [s] &  2.7 & 15.4 &  7.0 & 51.9 & 10.4\\
 \hline
 n=10000 & Fval &  210.7 &  210.7 &  210.7 &  210.7 &  210.7\\
 r=20 & Gradnorm & 1.0e-04 & 1.3e-06 & 9.1e-05 & 8.0e-05 & 9.8e-05 \\
  & Time [s] &  0.6 &  4.5 &  1.9 & 16.9 &  3.6\\
 \hline
 n=10000 & Fval & 2810.7 & 2810.7 & 2810.7 & 2810.7 & 2810.7\\
 r=50 & Gradnorm & 9.9e-05 & 9.9e-07 & 9.2e-05 & 8.7e-05 & 9.0e-05 \\
  & Time [s] &  3.7 & 31.3 &  9.2 & 59.0 & 24.3\\
 \hline
 \end{longtable}
 \end{footnotesize}
\end{center}

From the numerical results in Table 2, we can see that RDRSOM is more efficient than all the other algorithms in every instance. For some problems, RDRSOM is nearly 5 times faster than the second fastest algorithm. 

\subsection{K-means clustering}
In this subsection, we consider the following continuous formulation of the $k$-means clustering problem \cite{carson2017manifold,huang2022riemannian}:
\begin{equation}\label{kmeans}
\min\left\{ -\<WR,R\>+\lambda \|\Pi_{\R^{n\times K}_-}(R)\|^2:\ R^\top R=I_K,\ RR^\top e=e \right\},
\end{equation}
where $W=-YY^\top$ for some data matrix $Y\in \R^{n\times r}$ with $n$ samples and $r$ features. Here $\lambda>0$ is the penalty parameter that penalizes the the negative entries of 
$R$, and  $K\in \mathbb{N}$ is the number of clusters. We use datasets from the UCI Machine Learning Repository\footnote{Dataset from \href{https://archive.ics.uci.edu/ml/index.php}{https://archive.ics.uci.edu/ml/index.php}.}. We set $\lambda=100,$ and the maximum number of iterations of all algorithms to be $10^5.$ We terminate the algorithms when the norm of the Riemannian gradient is smaller than $10^{-4}.$

\begin{center}
\begin{footnotesize}
\begin{longtable}{lclclcccclll} 
\caption{Comparison of  RDRSOM, TR, CG, BFGS, BB methods for $k$-means clustering problems.}\\
\hline
problem & metric & RDRSOM & TR & CG& BFGS & BB  \\ \hline
 ecoli & Fval & -682.8 & -683.6 & -683.6 & -684.0 & -684.0\\
n=336 & Gradnorm & 9.6e-05 & 5.2e-06 & 8.4e-05 & 9.2e-05 & 9.2e-05 \\
K=5 & Time [s] & 0.1 & 0.7 & 0.3 & 1.1 & 0.6\\
\hline
ecoli & Fval & -689.2 & -689.2 & -689.2 & -689.4 & -689.9\\
n=336 & Gradnorm & 9.8e-05 & 3.7e-07 & 9.3e-05 & 9.9e-05 & 9.8e-05 \\
K=10 & Time [s] & 1.0 & 14.0 & 2.7 & 7.5 & 14.8\\
\hline
ecoli & Fval & -691.0 & -691.9 & -691.9 & -691.4 & -691.3\\
n=336 & Gradnorm & 9.6e-05 & 3.5e-05 & 9.7e-05 & 9.6e-05 & 9.9e-05 \\
K=20 & Time [s] & 2.8 & 54.0 & 5.2 & 17.7 & 7.5\\
\hline
yeast & Fval & -2102.8 & -2102.8 & -2102.8 & -2102.8 & -2102.8\\
n=1484 & Gradnorm & 9.2e-05 & 3.6e-05 & 7.6e-05 & 8.2e-05 & 7.4e-05 \\
K=5 & Time [s] & 0.1 & 1.2 & 0.6 & 2.5 & 0.7\\
\hline
yeast & Fval & -2115.7 & -2115.5 & -2115.6 & -2116.1 & -2115.8\\
n=1484 & Gradnorm & 8.8e-05 & 9.9e-05 & 9.9e-05 & 9.5e-05 & 9.8e-05 \\
K=10 & Time [s] & 1.4 & 8.3 & 2.6 & 12.8 & 5.3\\
\hline
yeast & Fval & -2121.8 & -2121.9 & -2121.9 & -2121.9 & -2121.5\\
n=1484 & Gradnorm & 8.8e-05 & 8.5e-05 & 9.6e-05 & 9.7e-05 & 1.0e-04 \\
K=20 & Time [s] & 4.1 & 60.1 & 12.6 & 59.9 & 36.5\\
\hline
segment & Fval & -8734.4 & -8734.4 & - & -8734.4 & -8734.4\\
n=2310 & Gradnorm & 6.3e-05 & 3.2e-08 & - & 5.5e-05 & 9.9e-05 \\
K=5 & Time [s] & 0.3 & 0.5 & - & 1.2 & 3.7\\
\hline
segment & Fval & -8950.6 & -8950.6 & -8950.6 & -8950.6 & -8950.6\\
n=2310 & Gradnorm & 9.7e-05 & 9.0e-05 & 9.2e-05 & 8.1e-05 & 9.5e-05 \\
K=10 & Time [s] & 0.3 & 0.9 & 0.8 & 3.2 & 8.3\\
\hline
segment & Fval & -9007.6 & -9009.8 & -9004.3 & -9003.0 & -9004.2\\
n=2310 & Gradnorm & 9.2e-05 & 9.4e-05 & 9.4e-05 & 9.7e-05 & 9.9e-05 \\
K=20 & Time [s] & 2.6 & 21.0 & 11.3 & 62.8 & 250.8\\
\hline
spambase & Fval & -2184.1 & -2184.1 & -2184.1 & -2184.1 & -2184.1\\
n=4601 & Gradnorm & 8.5e-05 & 8.9e-05 & 6.4e-05 & 8.6e-05 & 8.6e-05 \\
K=5 & Time [s] & 0.2 & 0.4 & 0.5 & 2.1 & 3.3\\
\hline
spambase & Fval & -2322.2 & -2337.9 & -2335.5 & -2335.5 & -2337.9\\
n=4601 & Gradnorm & 7.7e-05 & 8.9e-05 & 9.4e-05 & 9.0e-05 & 9.8e-05 \\
K=10 & Time [s] & 0.3 & 0.8 & 0.8 & 4.7 & 8.0\\
\hline
spambase & Fval & -2499.8 & -2499.8 & -2495.0 & -2492.2 & -2496.5\\
n=4601 & Gradnorm & 9.1e-05 & 3.1e-05 & 8.0e-05 & 8.5e-05 & 1.0e-04 \\
K=20 & Time [s] & 0.9 & 1.2 & 1.8 & 14.9 & 13.6\\
\hline
magic04 & Fval & -20209.3 & -20209.3 & -20209.3 & -20209.3 & -20209.3\\
n=19020 & Gradnorm & 9.4e-05 & 7.3e-05 & 1.5e-04 & 8.7e-05 & 9.4e-05 \\
K=5 & Time [s] & 0.5 & 0.5 & 0.5 & 3.5 & 5.3\\
\hline
magic04 & Fval & -20690.7 & -20690.7 & -20690.7 & -20690.7 & -20690.7\\
n=19020 & Gradnorm & 9.2e-05 & 9.4e-05 & 1.0e-04 & 9.3e-05 & 9.9e-05 \\
K=10 & Time [s] & 1.7 & 3.0 & 2.6 & 21.7 & 159.7\\
\hline
magic04 & Fval & -20735.7 & -20735.7 & -20735.6 & -20735.0 & -20731.7\\
n=19020 & Gradnorm & 9.5e-05 & 8.8e-05 & 1.4e-02 & 8.2e-05 & 1.4e-02 \\
K=20 & Time [s] & 68.3 & 186.6 & 83.5 & 648.0 & 888.4\\
\hline
 \end{longtable}
  \end{footnotesize}
\end{center}

From Table 3, we can see that RDRSOM can solve all the instances to the required accuracy. It is faster than the other algorithms except for the instance "magic04" with $K=5.$ Note that problem \cref{kmeans} is highly non-convex so the function values of the outputs of different algorithms might be different. The function value of RDRSOM can be either smaller or larger than the other algorithms (see "ecoli" $K=20$ and "spambase" $K=20$) but overall, they are at the same level. Again, CG method appears to be the second best performing algorithm in terms of computation times although it may fail to solve some instances such as 
``segment'' with $K=5.$
\subsection{Sensor network localization}
Now we consider the sensor network localization problem \cref{SNL}. To set up our experiments, we follow the settings used in \cite{biswas2006semidefinite} and choose $\lambda=0.1/n$ and $d=3$. We randomly generate $n$ sensors in the three dimensional box $[-0.5,0.5]\times [-0.5,0.5]\times [-0.5,0.5]$ using the 
code \texttt{Sen = rand(n,r)-0.5;} We choose $4$ anchors with locations $[0.3,-0.3,-0.3]$, $[-0.3,0.3,-0.3]$, $[-0.3,-0.3,0.3]$, and $[0.3,0.3,0.3].$ For sensors $i$ and $j,$ we add $(i,j)$ in $\N$ if their distance is smaller than or equal to $(45/4\pi n)^{1/3}$ so that the average degree of $\N$ is bounded by $15.$ For sensor $i$ and anchor $k,$ we add $(i,k)$ in $\G$ if the distance between $i$ and $k$ is smaller than or equal to $0.33\sqrt{2}$ so that one sensor may be adjacent to more than one anchors but
 it won't be adjacent to too many anchors. Let $m_1:=|\G|$ and $m_2:=|\N|.$ For any $(i,j)\in \N,$ let $d_{ij}$ be the distance between sensors $i$ and $j$, we add random noise to $d_{ij}$ as follows:
\begin{center}
\texttt{d = abs(1+0.2*randn(m2,1)).*d;}
\end{center}
We use the same initial point for the three algorithms, which is randomly generated as $\texttt{R = rand(n,r)-0.5;}$  We use different seeds so that the random initial point is different from the randomly generated sensor matrix. We set the maximum number of iterations and maximum running time of all algorithms to be 10000 and 3600s, respectively. We terminate our algorithm if the norm of the Riemannian gradient is smaller than $10^{-5}.$ We terminate the other two algorithms if the violation of optimality is smaller than $10^{-5}$ and the violation of primal feasibility is smaller than $10^{-6}.$ We use the RMSE$:=\|R-R_S\|/\sqrt{n}$ to measure the accuracy of the output, where $R_S$ is the computed location matrix of the sensors. We don't show the result of an algorithm if it reaches the maximum running time and the solution is still quite inaccurate.

\begin{center}
\begin{footnotesize}
\begin{longtable}{|c|c|ccccc|l|}
\caption{Comparison of RDRSOM and SQP and IPM for nonlinear sensor network localization. }
\\
\hline
problem & Algorithm & Fval & Gradnorm & Pfeas& RMSE & Time [s] \\ \hline
\endhead
$n=100$& RDRSOM & 8.0281847e-02 & 9.91e-06 & 8.20e-17 & 2.74e-01 & 1.08e+00 \\
$m_1=95$ & SQP & 8.9146352e-02 & 1.66e-05 & 3.72e-15 & 2.92e-01 & 3.78e+00 \\
$m_2=508$ & IPM & 9.2069471e-02 & 3.06e-05 & 4.49e-12 & 3.10e-01 & 2.88e+00 \\
\hline
$n=200$& RDRSOM & 1.1867660e-01 & 9.36e-06 & 9.97e-17 & 3.26e-01 & 1.47e+00 \\
$m_1=175$ & SQP & 1.1706289e-01 & 1.55e-05 & 1.44e-12 & 2.85e-01 & 7.94e+01 \\
$m_2=1132$ & IPM & 1.1870512e-01 & 5.16e-05 & 3.67e-12 & 2.84e-01 & 2.26e+01 \\
\hline
$n=300$& RDRSOM & 1.0896989e-01 & 9.54e-06 & 1.06e-16 & 2.67e-01 & 4.55e+00 \\
$m_1=251$ & SQP & 1.1487691e-01 & 5.16e-05 & 6.00e-11 & 3.23e-01 & 2.72e+01 \\
$m_2=1856$ & IPM & 1.0360072e-01 & 6.43e-05 & 4.26e-10 & 2.26e-01 & 7.92e+01 \\
\hline
$n=400$& RDRSOM & 9.5185627e-02 & 1.00e-05 & 2.09e-16 & 2.77e-01 & 3.46e+00 \\
$m_1=341$ & SQP & 1.0500715e-01 & 5.23e-05 & 5.77e-13 & 3.41e-01 & 7.43e+01 \\
$m_2=2481$ & IPM & 9.2686891e-02 & 6.01e-05 & 4.19e-11 & 2.50e-01 & 1.04e+02 \\
\hline
$n=500$& RDRSOM & 8.6091412e-02 & 9.96e-06 & 3.83e-16 & 3.20e-01 & 5.28e+00 \\
$m_1=435$ & SQP & 8.7008298e-02 & 5.22e-05 & 3.10e-13 & 3.55e-01 & 5.34e+02 \\
$m_2=3031$ & IPM & 8.3023835e-02 & 7.81e-05 & 3.07e-09 & 2.63e-01 & 2.12e+02 \\
\hline
$n=1000$& RDRSOM & 8.0562748e-02 & 9.97e-06 & 1.91e-15 & 3.43e-01 & 1.03e+01 \\
$m_1=858$ & SQP & 8.1228996e-02 & 3.45e-04 & 2.76e-10 & 3.34e-01 & 2.60e+03 \\
$m_2=6467$ & IPM & 8.0534113e-02 & 9.12e-05 & 1.23e-10 & 3.34e-01 & 1.61e+03 \\
\hline
$n=2000$& RDRSOM & 5.4523129e-02 & 9.92e-06 & 1.62e-15 & 2.34e-01 & 6.84e+01 \\
$m_1=1702$ & SQP & - & - & - & - & - \\
$m_2=13272$ & IPM & - & - & - & - & - \\
\hline
$n=3000$& RDRSOM & 5.6133020e-02 & 9.85e-06 & 2.89e-15 & 3.27e-01 & 4.92e+01 \\
$m_1=2541$ & SQP & - & - & - & - & - \\
$m_2=20025$ & IPM & - & - & - & - & - \\
\hline
$n=5000$& RDRSOM & 4.8662073e-02 & 9.85e-06 & 1.97e-15 & 2.91e-01 & 1.76e+02 \\
$m_1=4186$ & SQP & - & - & - & - & - \\
$m_2=34274$ & IPM & - & - & - & - & - \\
\hline
$n=10000$& RDRSOM & 3.5716471e-02 & 9.94e-06 & 2.44e-15 & 2.96e-01 & 5.10e+02 \\
$m_1=8309$ & SQP & - & - & - & - & - \\
$m_2=69408$ & IPM & - & - & - & - & - \\
\hline
\end{longtable}
\end{footnotesize}
\end{center}

From Table 4, we can see that RDRSOM can solve all the problems to the required accuracy while SQP and IPM cannot solve problems of size beyond $1000.$ Moreover, RDRSOM is much more efficient than the other two algorithms. For some instances, RDRSOM is more than 100 times faster than the other two algorithms. This verifies the efficiency of our Riemannian optimization method in solving problem \cref{SNL}. Note that the residue of the optimality condition in \texttt{fmincon} is different from Riemannian gradient. Also, \texttt{fmincon} will terminate if the stepsize is too small. Thus, the norm of the Riemannian gradient of the output from SQP and IPM may not reach the accuracy of $10^{-6}.$
 Although RDRSOM can return a solution of higher accuracy than SQP and IPM, the RMSEs of the three algorithms are at the same level. One reason is that we have added noise to the distance measurements between different sensors. Thus, solving problem \cref{SNL} accurately doesn't imply exact estimation of the sensors' locations. Another reason is that problem \cref{SNL} is highly non-convex, the first order optimality cannot totally measure the quality of the solution. In order to improve the quality of the solution of \eqref{SNL}, we will warm start it by solving \cref{SNL1}. According to the classical rank bound for SDP problem \cite{PG}, problem \cref{SNLSDP} is equivalent to \cref{SNL} when $r\geq \sqrt{2\(|\N|+|\G|+d(d+1)/2\)}.$ Nonetheless, given the fact that all coefficient matrices of \cref{SNLSDP} are symmetric positive semidefinite when $\lambda=0$, we can employ an enhanced rank bound of
  $\O\(\log (|\N|+|\G|)\)$ by So, Ye and Zhang in \cite{so2008unified} and simply choose $r=20$ in \cref{SNL1}. In detail, we use RDRSOM to solve \cref{SNL1} until the norm of the 
  Riemannian gradient of $\hat{R}$ is smaller than $10^{-3}.$ Then we use the first $r$ columns of $\hat{R}$ to be the initial point of RDRSOM for solving \cref{SNL}. In the second stage, we use the same problem setting as mentioned before, with the only difference being in the initialization. We report the running time be the total running time of these two stages.

\begin{center}
\begin{footnotesize}
\begin{longtable}{|c|c|ccccc|l|}
\caption{Solving \cref{SNL} with initialization from \cref{SNL1}. }
\\
\hline
problem & Algorithm & Fval & Gradnorm & Pfeas& RMSE & Time [s] \\ \hline
\endhead
$n=100$& RDRSOM & 5.6524983e-02 & 9.30e-06 & 2.02e-15 & 9.06e-02 & 9.34e-01 \\
$n=200$& RDRSOM & 8.6382506e-02 & 9.10e-06 & 2.14e-15 & 9.02e-02 & 1.90e+00 \\
$n=300$& RDRSOM & 8.0980859e-02 & 9.57e-06 & 3.30e-15 & 7.13e-02 & 5.22e+00 \\
$n=400$& RDRSOM & 6.1333830e-02 & 9.96e-06 & 1.10e-16 & 5.15e-02 & 3.01e+00 \\
$n=500$& RDRSOM & 5.5472952e-02 & 9.63e-06 & 2.51e-15 & 6.27e-02 & 4.75e+00 \\
$n=1000$& RDRSOM & 4.4575380e-02 & 9.93e-06 & 3.48e-15 & 4.02e-02 & 7.72e+00 \\
$n=2000$& RDRSOM & 3.1852253e-02 & 9.93e-06 & 2.25e-15 & 3.01e-02 & 1.63e+01 \\
$n=3000$& RDRSOM & 2.6631692e-02 & 9.70e-06 & 2.93e-15 & 2.63e-02 & 2.37e+01 \\
$n=5000$& RDRSOM & 2.0595260e-02 & 9.99e-06 & 2.61e-15 & 2.12e-02 & 4.23e+01 \\
$n=10000$& RDRSOM & 1.1507971e-02 & 9.84e-06 & 3.25e-15 & 1.68e-02 & 1.22e+02 \\
\hline
 \end{longtable}
\end{footnotesize}
\end{center}

Comparing Table 4 and Table 5, we can see that the warm-start strategy significantly improves the solution of the nonlinear model. Specifically, the function value has decreased significantly and the root mean square errors (RMSEs) associated with the new method are all below $0.1$. Moreover, the new approach with warm-start strategy is even faster than previous single-stage method for most test instances. This is because our low rank SDP model \eqref{SNL1} alleviates the non-convexity of \eqref{SNL} while maintaining small dimensionality. These results provide empirical evidence for the efficiency and effectiveness of the low-rank SDP formulation \eqref{SNL1} for solving \eqref{SNL}.

\section{Conclusion}
In this paper, we proposed a cubic-regularized Riemannian dimension reduced second order method RDRSOM, which partially exploits the second order information. We establish the iteration
complexity of $\O(1/\epsilon^{3/2})$ for a theoretical version of RDRSOM, 
and the complexity of $\O(1/\epsilon^2)$ for a more 
practical version of RDRSOM, where the latter put less restriction on the subspace 
chosen for the subproblem.
 We apply our algorithm to solve a nonlinear formulation of the sensor network localization problem. The efficiency of RDRSOM is clearly demonstrated in numerical experiments 
 as compared to other Riemannian optimization methods and nonlinear solvers.


\bibliographystyle{siamplain}
\bibliography{rdrsom-2}
\end{document}